\documentclass[12pt,onecolumn,draftcls]{IEEEtran}

\usepackage[usenames,dvipsnames]{xcolor}
\usepackage{setspace}
\usepackage[draft]{hyperref}
\usepackage{cite}
\usepackage{amsmath,amssymb,euscript,yfonts,psfrag,latexsym,dsfont,graphicx,bbm,color,amstext,wasysym,balance,mathtools}
\usepackage{algorithm,algpseudocode}
\usepackage{graphicx}
\usepackage{amsthm}
\usepackage{float,subfig}

 \usepackage{multirow}

\makeatletter
\def\footnoterule{\relax%
  \kern-5pt
  \hbox to \columnwidth{\hfill\vrule width 0.5\columnwidth height 0.4pt\hfill}
  \kern4.6pt}
\makeatother

\graphicspath{{../figures/}{figures/}{./figures/}}

\newcommand{\mR}{{\mathbb R}}

\newtheorem{lemma}{Lemma}
\newtheorem{cor}{Corollary}
\newtheorem{prop}{Proposition}

\newtheorem{thm}{Theorem}

\newcommand{\bv}{{\boldsymbol v}}

\newcommand{\bx}{{\boldsymbol x}}

\newcommand{\by}{{\boldsymbol y}}

\newcommand{\btheta}{{\boldsymbol \theta}}

\newcommand{\mS}{{\mathbb S}}
\newcommand{\ma}{{\rm a}}
\newcommand{\ms}{{\rm s}}

\newcommand{\cSym}{{S_{ym}}}

\newcommand{\cE}{{\mathcal E}}

\newcommand{\cP}{{\mathcal P}}

\newcommand{\dd}{{\operatorname{d}}}
\newcommand{\g}{{\operatorname{g}}}
\newcommand{\f}{{\operatorname{f}}}
\newcommand{\w}{{\operatorname{w}}}
\newcommand{\trace}{{\operatorname{tr}}}

\newcommand{\argmin}{{\operatorname{argmin}}}

\begin{document}

\title{Regularization of covariance matrices on Riemannian manifolds using linear systems}
\author{Lipeng Ning \thanks{L. Ning is with the Department of Psychiatry, Brigham and Women's Hospital, Harvard Medical School. Email: lning@bwh.harvard.edu} }
	
\maketitle
\date{}

\begin{abstract}
We propose an approach to use the state covariance of linear systems to track time-varying covariance matrices of non-stationary time series. 
Following concepts from Riemmanian geometry, we investigate three types of covariance paths obtained by using different quadratic regularizations of system matrices.
The first quadratic form induces the geodesics based on the Bures-Wasserstein metric from optimal mass transport theory and quantum mechanics. 
The second type of quadratic form leads to the geodesics based on the Fisher-Rao metric from information geometry. 
In the process, we introduce a fluid-mechanics interpretation of the Fisher-Rao metric for multivariate Gaussian distributions. 
A main contribution of this work is the introduction of the third type of covariance paths which are steered by linear system matrices with rotating eigenspace. 
We provide theoretical results on the existence and uniqueness of this type of covariance paths. 
The three types of covariance paths are compared using two examples with synthetic data and real data based on functional magnetic resonance imaging, respectively. 
\end{abstract}


\IEEEpeerreviewmaketitle

\section{Introduction}
The problem of tracking changes and deformations of positive definite matrices is relevant to a wide spectrum of scientific applications, including computer vision, sensor array, and diffusion tensor imaging, see e.g. \cite{PorikliCVPR,WuTIP2012,Yang1988,JNG2012,Lenglet2006,Dryden2009,Hao2011}. 
A key motivation behind the present work is from a neuroscience application on understanding functional brain connectivity using resting-state function MRI (rsfMRI) data.
Specifically, rsfMRI is a widely used neuroimaging modality which acquires a sequence three-dimensional image volumes from the whole brain to understand brain functions and activities \cite{Biswal1995}.
The standard approach for analyzing the interdependence between brain regions is to compute the correlation coefficient between the underlying rsfMRI time-series data.
Typically, the whole-brain functional network is characterized by the covariance matrix of a multivariate time-series data obtained from different brain regions \cite{Buckner2013,Smith2013}. 
It has been recently observed that the functional connectivity fluctuates over time \cite{CHANG201081,PRETI201741}, implying that the static covariance matrix based on the assumption of stationary time series may be too simplistic to capture the full extent of brain activities. 
Thus there is an urgent need for new computational tools for understanding dynamic functional brain networks using non-stationary rsfMRI data. 

The aim of this work is use control-theoretic approaches to develop models for time-varying covariance matrices. In particular, we consider a non-stationary, discrete-time random process as observations of a zero-mean time-dependent random variable $\bx_t\in \mR^n$. 
We assume that the temporal change of the probability distributions $p_t(\bx)$ of $\bx_t$ is much slower than the rate of measurements. 
Therefore, the instantaneous covariance matrix
\[
P_t:=\cE_{p_t}(\bx \bx')=\int_{\mR^n} \bx \bx' p_t(\bx) d\bx,
\]
can be estimated by using sample covariance matrices computed from measurements from short time windows.
Assume that two covariance matrices $P_0$ and $P_1$ at $t=0, 1$ are known. 
Then geodesics connecting $P_0$ and $P_1$ on the manifold of positive-definite matrices provide natural structures to model time-varying covariance matrices $P_t$ on the time interval $t\in[0, 1]$.
As generalizations of straight lines in Euclidean space, geodesics are paths of shortest distances connecting the start to the finish on a curved manifold.   
The path length is measured by Riemmanian metrics which are quadratic forms of the tangent matrix $\dot P_t$. 
Several Riemannian metrics have been proposed to derive geodesics for covariances matrices. 
For instance, the geodesics based on the Fisher-Rao metric for Gaussian distributions \cite{Rao1945,Amari2000,Cencov1982,Kass1997} from the theory of information geometry is given by
\begin{align}\label{eq:Ptinfo}
P_{t}^{\rm info}=P_0^{1/2}(P_0^{-1/2} P_1P_0^{-1/2} )^t P_0^{1/2}.
\end{align} 
An other example would be the Wasserstein-2 metric for Gaussian probability density functions \cite{VillaniBook,RachevBook,Knott1984,Takatsu2008}, which induces the following geodesic
\begin{align}\label{eq:Ptomt}
P^{\rm omt}_{t}=\big((1-t)P_0^{\tfrac12}+tP_1^{\tfrac12} \hat U\big)\big((1-t)P_0^{\tfrac12}+tP_1^{\tfrac12} \hat U\big)',
\end{align}
where $P_t^{\tfrac12}$ denotes the unique positive-definite square root of $P_t$, and 
\begin{align*}
\hat U=P_1^{-1/2}P_0^{-1/2} (P_0^{1/2}P_1P_0^{1/2})^{1/2}
\end{align*} 
is an orthogonal matrix.
These Riemmanian metrics will be explained in more details in the following sections.

In this paper, we combine concepts from Riemmanian geometry and linear systems to develop smooth covariance paths. Specifically, we consider a geodesic $P_t$ as the state covariance of the following linear system 
\begin{align}\label{eq:Ax}
\dot \bx_t=A_t \bx_t,
\end{align}
with $A_t\in \mR^{n\times n}$. We analyze the linear systems that steer $P_t$ along different geodesics.
Based on \eqref{eq:Ax}, the state covariance evolves according to
\begin{align}\label{eq:dotP}
\dot P_t=A_tP_t+P_tA_t'.
\end{align}
Thus, the matrix $A_t$ can be considered as a non-commutative devision of $\dot P_t$ by $P_t$ scaled by a factor of $\tfrac12$.
Clearly, if $P_t$ is given, the mapping from $A_t$ to $\dot P_t$ induced by \eqref{eq:dotP} is injective. 
In this paper, we consider the $A_t$ that the minimizer of a quadratic function $f(A_t)$ such that \eqref{eq:dotP} holds.
The optimal value $f(A_t)$ provides an alternative way to define Riemannian metrics for measuring the length of covariance paths.
To this end, we consider covariance paths that are the solutions to
\begin{align}\label{eq:minA}
\min_{P_t, A_t} \left\{\int_0^1 \f(A_t)dt \mid \dot P_t=A_t P_t+P_tA_t', P_0, P_1 \mbox { given}\right\}.
\end{align}
Note that the optimal system matrix $A_t$ may not be symmetric, which could provide useful information to understand directed interactions among the underlying variables.
In this work, we investigate the solution to \eqref{eq:minA} with three different types of quadratic forms of $f(A_t)$. 
The first two quadratic forms lead to the geodesic paths $P_t^{\rm omt}$ and $P_t^{\rm info}$, respectively.
A main contributions of this work in the introduction of the third family of geodesics which are the state covariances of linear system with rotating eigespace.

This paper is organized as follows. 
In Section \ref{sec:omt}, we revisit the OMT-based geodesics $P_{t}^{\rm omt}$ and introduce the corresponding quadratic form $f(A_t)$.
Section \ref{sec:info} will focus on the quadratic forms that lead to the Fisher-Rao-based geodesics $P_{t}^{\rm info}$. 
We also introduce a fluid-mechanics interpretation of the Fisher-Rao metric, which provides a point of contact between OMT and information geometry.
In Section \ref{sec:WLS}, we investigate the optimal solutions to \eqref{eq:minA} corresponding to a family of quadratic functions $f(A_t)$ which are weighted-square norms of the symmetric and asymmetric part of $A_t$. 
We provide the expression of the optimal solutions and analysis on the existence and uniqueness of the solutions.
In Section \ref{sec:example}, we compare the three types of covariances paths using two examples based on synthetic data and real data from rsfMRI, respectively.
Section \ref{sec:discussion} includes the discussions and conclusions.

For notations, $\mS^{n}, \mS^{n}_{+}, \mS^{n}_{++}$ denote the set of symmetric, positive semidefinite, and positive definite matrices of size $n\times n$, respectively. 
Small boldface letters, e.g. $\bx, \bv$, represent column vectors. Capital letters, e.g. $P, A$, denote matrices. Regular small letters, e.g. $\w, h$ are for scalars or scalar-valued functions.

\section{Mass-transport based covariance paths}\label{sec:omt}
\subsection{On optimal mass transport}
Let $p_0(\bx)$ and $p_1(\bx)$ denote two probability density functions on $\mR^n$. The Wasserstein-2 metric between the two, denoted by $\w_2(p_0, p_1)$, is defined by
\begin{align*}
\w_2(p_0, p_1)^2=&\inf_{m(\bx,\by)\geq0}  \int_{\mR^n\times \mR^n} \|\bx-\by\|_2^2 m(\bx,\by)d\bx d\by,   \nonumber\\
\mbox{s.t. }&  \int_{\mR^n} m(\bx,\by)d\bx=p_1(\by), \int_{\mR^n} m(\bx,\by)d\by=p_0(\bx),
\end{align*}
where $m(\bx,\by)$ represents a probability density function on the joint space $\mR^n\times \mR^n$ with the marginals specified by $p_0$ and $p_1$, \cite{VillaniBook,RachevBook}. A fluid-mechanics interpretation of $\w_2(p_0, p_1)^2$ was introduced in \cite{JKO1998,Benamou2000}, which provided a Riemmanian structure of the manifold of probability density functions. 
To introduce this formula, we consider the following continuity equation 
\begin{align}\label{eq:continuity}
\frac{\partial p_t(\bx)}{\partial t}+\nabla_\bx\cdot(p_t(\bx) \bv_t(\bx))=0,
\end{align}
where $\bv_t(\bx)$ represents a time-varying velocity field.
Then, $\w_2(p_0, p_1)^2$ is equal to \cite{Benamou2000}
\begin{align}\label{eq:Benamou}
\w_2(p_0, p_1)^2=\inf_{p_t,\bv_t} \bigg\{\int_0^1 \cE_{p_t}(\|\bv_t(\bx)\|_2^2)dt \mid&  \dot p_t+\nabla_\bx\cdot(p_t \bv_t)=0 \bigg\}.
\end{align}
The optimal solution of $p_t(\bx)$ is the geodesic on the manifold of probability density functions that connects the endpoints $p_0(\bx)$ and $p_1(\bx)$.

\subsection{The Bures-Wasserstein metric}
In the special case when $p_0(\bx)$ and $p_1(\bx)$ are zero-mean Gaussian probability density functions with
\[
p_i(\bx)=\det(2\pi P_i)^{-\frac12}e^{(-\tfrac12 \bx'P_i^{-1} \bx)}, \mbox{ for } i=0, 1,
\] 
the corresponding geodesic $p_t(\bx)$ at any fixed time $t$ is also zero-mean Gaussian with the corresponding covariance matrix given by $P_t^{\rm omt}$ \cite{Knott1984,Takatsu2008}. 
The geodesic distance is equal to Wasserstein-2 metric $\w_2(p_0, p_1)$, which also induces the following distance measure on the covariance matrices
\begin{align}
\w_2(p_0,p_1)=\dd_{{\rm w}_2}(P_0, P_1):=\|P_0^{\tfrac12}-P_1^{\tfrac12}\hat U\|_{\rm F},
\end{align}
where $\|\cdot\|_{\rm F}$ denotes the Frobenius norm of a matrix.

The covariance path $P_t^{\rm omt}$ is also equal to the geodesic induced by the Bures metric from quantum mechanics \cite{Ning2013,Bhatia2017} on the manifold of positive definite matrices.
In particular, let $P\in \mS^{n}_{++}$ and $\Delta\in \cSym_n$ which represents a tangent vector at $P$. 
The Bures metric takes the form
\[
\g_{P,{\rm Bures}}(\Delta)=\trace(\Delta M),
\] 
where $M\in \mS^{n}$ and $\frac12(PM+MP)=\Delta$,
see e.g. \cite{Uhlmann1992,Petz1994}.
The trajectory $P_t^{\rm omt}$ in \eqref{eq:Ptomt} is the shortest path connecting $P_0$ and $P_1$.
Thus, it satisfies that 
\begin{align}\label{eq:PtBures}
P_t^{\rm omt}=\underset{P_t}{\argmin} \int_0^1 \g_{P_t,{\rm Bures}}(\dot P_t)dt ,
\end{align}
with a given pair of endpoints $P_0, P_1\in \mS^n_{++}$, see e.g. \cite{Ning2013}.
The geodesic distance, also known as the Bures distance, is equal to $\dd_{{\rm w}_2}(P_0, P_1)$.

The Bures metric was originally proposed in quantum mechanics to compare density matrices, which are positive definite matrices whose traces are equal to one. 
The density matrices are non-commutative analogues of probability vectors. 
In the commutative case when $P$ is restricted to be a diagonal matrix whose diagonal entries consist of a probability vector, 
then the Bures metric $\g_{P,{\rm Bures}}(\Delta)$ is equal the Fisher information metric which will be discussed in Section \ref{sec:info}.

\subsection{Bures-Wasserstein-based linear systems}
Here, we present an alternative expression of \eqref{eq:PtBures} using the linear system in \eqref{eq:Ax}. For this purpose, we define that
\begin{align}\label{eq:fPt}
\f_{P_t}(A_t):=\trace(A_tP_tA_t'),
\end{align}
which is equal to $\cE_{p_t}(\|\dot\bx_t\|^2)$ if $\dot \bx_t$ is given by \eqref{eq:Ax}. The following proposition relates the geodesics $P_t^{\rm omt}$ to a linear system.
\begin{prop}\label{thm:transp}
Given $P_0, P_1\in \mS^{n}_{++}$. Then $P_{t}^{\rm omt}$ given by \eqref{eq:Ptomt} is the unique minimizer of
\begin{align}\label{eq:minA_transp}
\underset{P_t,A_t}{\min} \left\{ \int_0^1 \f_{P_t}(A_t)dt \mid \dot P_t=A_t P_t+P_tA_t', P_0, P_1 \mbox{ specified}\right\},
\end{align}
with $\f_{P_t}(A_t)$ defined by \eqref{eq:fPt} 
and the optimal $A_t$ is equal to
\begin{align}\label{eq:Aomt}
A_{t}^{\rm omt}= Q(tQ-I)^{-1},
\end{align}
where 
$Q=I-P_0^{-\tfrac12}(P_0^{\tfrac12}P_1P_0^{\tfrac12})^{\tfrac12} P_0^{-\tfrac12}$.
The optimal value of the objective function in \eqref{eq:minA_transp} is equal to $\dd_{{\rm w}_2}(P_0, P_1)^2$.
\end{prop}
\begin{proof}
The optimization problem \eqref{eq:minA_transp} takes a special form of \eqref{eq:Benamou} with $p_0(\bx), p_1(\bx)$ being two zero-mean Gaussian probability density functions and an additional constraint that the velocity field $\bv_t(\bx)=A_t\bx$. 
Thus, $\dd_{{\rm w}_2}(P_0, P_1)^2$ is a lower bound of \eqref{eq:minA_transp}. Therefore, we only need to show that $A_{t}^{\rm omt}$ satisfies the constraint and provides the optimal value. 
\begin{align}\label{eq:dotPomt}
\dot P_{t}^{\rm omt}= A_{t}^{\rm omt} P_{t}^{\rm omt}+ P_{t}^{\rm omt}(A_{t}^{\rm omt})'.
\end{align}

First, we rewrite \eqref{eq:Ptomt} as
$P_{t}^{\rm omt}=(I-tQ) P_0(I-tQ)$.
Taking the derivative of $P_{t}^{\rm omt}$ gives
\begin{align*}
\dot P_{t}^{\rm omt}=&-Q P_0(I-tQ)-(I-tQ)P_0Q\\
=&~Q(tQ-I)^{-1}P_t+P_t Q(tQ-I)^{-1}.
\end{align*}
Since all the eigenvalues of $Q$ are smaller than $1$, the matrix $tQ-I$ is invertible for all $t\in [0, 1]$.
Therefore \eqref{eq:dotPomt} holds.
Moreover,
\begin{align*}
&\trace(A_{t}^{\rm omt} P_{t}^{\rm omt} A_{t}^{\rm omt})=\trace(QP_0Q)\\
&=\|P_1^{1/2}\hat U-P_0^{1/2}\|_{\rm F}^2=\dd_{{\rm w}_2}(P_0,P_1)^2,
\end{align*}
which completes the proof.
\end{proof}

We note that the matrix $A_t^{\rm omt}$ is symmetric. Moreover, the matrices $A_{t_1}^{\rm omt}$ and $A_{t_2}^{\rm omt}$ commute for any $t_1, t_2$. Therefore the eigenspace of $A_t$ is fixed on the interval $t\in[0, 1]$.

The results from Proposition \eqref{thm:transp} can be further extended to obtain the optimal solutions corresponding to the following objective function
\begin{align}\label{eq:fPtW}
\f_{P_t}^W(A_t)=\cE_{p_t}(\|\dot \bx_t\|_W^2)=\trace(WA_tP_tA_t'),
\end{align}
where $\dot \bx_t=A_t \bx$, $W\in \mS^{n}_{++}$ and $\|\dot \bx_t\|_W^2:=\bx' W\bx$. 
By applying change of variables, we define 
\begin{align}
P_{W,t}&:=W^{\tfrac12}P_tW^{\tfrac12},\label{eq:PW}\\
A_{W,t}&:=W^{\tfrac12}A_tW^{-\tfrac12}\label{eq:AW}.
\end{align}
Thus, $\f_{P_t}^W(A_t)=\trace(A_{W,t} P_{W,t} A_{W,t}')=\f_{P_{W,t}}(A_{W,t})$. Moreover, if \eqref{eq:dotP} holds, then 
\[
\dot P_{W,t}=A_{W,t} P_{W,t}+P_{W,t}A_{W,t}'.
\]
Therefore, if $A_{t}^{\rm omt}$ is the optimal system matrix that steers $P_{W,0}$ to $P_{W,1}$ with respect to $\f_P(A)$ given by Proposition \ref{thm:transp}, then $W^{-\tfrac12}A_{t}^{\rm omt}W^{\tfrac12}$ is the optimal solution with respect to $\f_P^W(A)$.
In the following section, we investigate a further extension of $\f_{P}^W(\cdot)$ by using a time-dependent weighting matrix which provides a point of contact between OMT and information geometry.

\section{Information-geometry based covariance paths}\label{sec:info}

\subsection{The Fisher-Rao metric}
For two probability density functions $p(\bx)$ and $\hat p(\bx)$ on $\mR^n$, the Kullback-Leibler (KL) divergence 
\begin{align}\label{eq:DKL}
\dd_{\rm KL}(p|| \hat p):=\int_{\mR^n} p \log\left(\frac{p}{\hat p} \right)d\bx 
\end{align}
represents a well-established notion of distance between the two \cite{KL1951,CoverBook}. If $\hat p=p+\delta$ with $\delta$ representing a small perturbation, then the quadratic term of the Taylor's expansion of $\dd_{\rm KL}(p || p+\delta)$ in terms of $\delta$ is the Fisher information metric 
\[
\g_{p,{\rm Fisher}}(\delta)=\int \frac{\delta^2}{p}d\bx.
\]
For a probability distribution $p(\bx,\btheta)$ parameterized by a vector $\btheta$, the corresponding metric is referred to as the Fisher-Rao metric and given by
\[
\g_{\btheta,{\rm Rao}}(\delta_{\btheta})=\delta_{\btheta}'\cE\left[\left(\frac{\partial \log p}{\partial \btheta} \right)\left(\frac{\partial \log p}{\partial \btheta} \right)' \right]\delta_{\btheta}.
\]
If $p(\bx)$ is a zero-mean Gaussian probability density function parameterized by a covariance matrix $P$, then the metric becomes
\[
\g_{P,{\rm Rao}}(\Delta)=\trace(P^{-1}\Delta P^{-1}\Delta).
\]
Given $P_0, P_1\in \mS^{n}_{++}$, the geodesic in \eqref{eq:Ptinfo} is equal to the solution
\begin{align}\label{eq:Ptinfo_min}
P_{t}^{\rm info}=\underset{P_t}{\argmin} \int_0^1 \g_{P,{\rm Rao}}(\dot P_t)dt,
\end{align}
which is the shortest path connecting $P_0$ and $P_1$.
The corresponding path length is equal to
\begin{align}\label{eq:Distinfo}
\dd_{\rm info}(P_0, P_1)=\|\log(P_0^{-\tfrac12}P_1P_0^{-\tfrac12}) \|_{\rm F},
\end{align}
see e.g. \cite[Theorem 6.1.6]{Bhatia2007}

\subsection{Fisher-Rao metric based linear systems}
Following \eqref{eq:minA}, we will define a positive quadratic form so that the optimal state covariance is equal to $P_t^{\rm info}$. 
One choice for the quadratic form would be given by 
\begin{align} 
\f_{P}^{{\rm info},1}(A):&=\g_{P,{\rm Rao}}(A P+PA')\nonumber\\
&=2\trace(AA+P^{-1}APA'),\label{eq:fPinfo1}
\end{align}
which satisfies that $\f_{P}^{{\rm info},1}(A)\geq0$, $\forall A\in \mR^{n\times n}$ and $P\in \mS^{n}_{++}$.  
\begin{prop}\label{thm:info1}
Given $P_0, P_1\in \mS^{n}_{++}$, $P_t^{\rm info}$ from \eqref{eq:Ptinfo} is the unique minimizer of
\begin{align}\label{eq:min_info1}
\underset{P_t, A_t}{\min}\bigg\{ \int_0^1 \f_{P_t}^{{\rm info},1}(A_t)dt \mid& \dot P_t=A_tP_t+P_tA_t', P_0, P_1 \mbox{ specified}\bigg\},
\end{align}
with $\f_{P_t}^{\rm info}(A_t)$ defined by \eqref{eq:fPinfo1} 
and the optimal $A_t$ is equal to
\begin{align}\label{eq:At_info}
A^{\rm info}:= \frac12 P_0^{\tfrac12}\log(P_0^{-\tfrac12} P_1P_0^{-\tfrac12}) P_0^{-\tfrac12}.
\end{align}
Moreover, the optimal value of the objective function is equal to $\dd_{\rm info}(P_0, P_1)^2$. 
\end{prop}
\begin{proof}
We rewrite \eqref{eq:Ptinfo} as
\begin{align*}
P_{t}^{\rm info}&=P_0^{\tfrac12}(P_0^{-\tfrac12} P_1P_0^{-\tfrac12})^t P_0^{\tfrac12}\\
&=P_0^{\tfrac12}e^{\tfrac12 \log(P_0^{-\tfrac12} P_1P_0^{-\tfrac12})t }e^{\tfrac12 \log(P_0^{-\tfrac12} P_1P_0^{-\tfrac12})t }P_0^{\tfrac12}\\
&=e^{A^{\rm info}t}P_0e^{{A^{\rm info}}'t},
\end{align*}
where the last equation is obtained using $e^{XYX^{-1}}=Xe^YX^{-1}$.
Taking the derivative of $P_{t}^{\rm info}$ to give that
\begin{align}\label{eq:AtPt_info}
\dot P_t^{\rm info}=A^{\rm info}P_t^{\rm info}+P_t^{\rm info} {A^{\rm info}}',
\end{align}
which completes the proof.
\end{proof}

Note that the metric $\dd_{\rm info}(\cdot)$ is invariant with respect to congruence transforms, i.e. $\dd_{\rm info}(P_0, P_1)=\dd_{\rm info}(TP_0T', TP_1T')$ for any invertible matrix $T$. If $A_{t}^{\rm info}$ is the optimal solution of \eqref{eq:min_info1}. Then $TA^{\rm info}T^{-1}$ is the optimal solution corresponding to the pair $TP_0T', TP_1T'$.

\subsection{A weighted-mass-transport view}
Since the map from $A_t$ to $\dot P_t$ in \eqref{eq:dotP} is injective, the quadratic forms of $A_t$ that lead to $P_t^{\rm info}$ is not unique. 
Here, we provide an alternative form, which provides an interesting relation between OMT and the Fisher-Rao metric. 
For this purpose, we define the following weighted square norm
\begin{align}\label{eq:fPinfo2}
\f_{P_t}^{{\rm info},2}(A_t)=4\cE_{p_t}(\|\dot \bx_t\|_{P_t^{-1}}^2)=4\trace(P_t^{-1}A_tP_t A_t'),
\end{align}
which is a special form of \eqref{eq:fPtW} with the weighting matrix $W=P_t$ and scaled by the factor of $4$. 
The following lemma draws the relation between $\f_{P}^{{\rm info},2}(A)$ and $\f_{P}^{{\rm info},1}(A)$.
\begin{lemma}\label{lemma:fPinfo}
Consider $\f_{P}^{{\rm info},1}(\cdot)$ and $\f_{P}^{{\rm info},2}(\cdot)$ be defined in \eqref{eq:fPinfo1} and \eqref{eq:fPinfo2}, respectively. Then,
\[
\f_{P}^{{\rm info},2}(A)\geq \f_{P}^{{\rm info},1}(A), \forall P\in \mS^{n}_{++}, A\in \mR^{n\times n}.
\]
\end{lemma}
\begin{proof}
Taking the difference
\begin{align*}
&\f_{P}^{{\rm info},2}(A)-\f_{P}^{{\rm info},1}(A)\\
=&~2\trace(P^{-1}APA'-AA)\\
=&~\trace\left((P^{-\tfrac12}AP^{\tfrac12}-P^{\tfrac12}A'P^{-\tfrac12})(P^{-\tfrac12}AP^{\tfrac12}-P^{\tfrac12}A'P^{-\tfrac12})'\right),
\end{align*}
which is non-negative.
\end{proof}
We note that if $P^{-\tfrac12}AP^{\tfrac12}$ is symmetric, then $\f_{P}^{{\rm info},1}(A)$ is equal to $\f_{P}^{{\rm info},2}(A)$. This gives rise to the following proposition in parallel to Proposition \ref{thm:info1}.
\begin{prop}\label{prop:info2}
Given $P_0, P_1\in \mS^{n}_{++}$. Then $P_t^{\rm info}, A^{\rm info}$ given by \eqref{eq:Ptinfo} and \eqref{eq:At_info}, respectively, are the unique pair of minimizer of
\begin{align}\label{eq:min_info2}
\underset{P_t, A_t}{\min}\bigg\{ \int_0^1 \f_{P_t}^{{\rm info},2}(A_t)dt \mid& \dot P_t=A_tP_t+P_tA_t', P_0, P_1 \mbox{ specified}\bigg\},
\end{align}
with $\f_{P_t}^{{\rm info},2}(A_t)$ defined by \eqref{eq:fPinfo2}.
\end{prop}
\begin{proof}
From Lemma \ref{lemma:fPinfo}, 
\begin{align}
\int_0^1 \f_{P_t}^{{\rm info},2}(A_{t})dt \geq \int_0^1 \f_{P_t}^{{\rm info},1}(A_{t})dt\geq \dd_{\rm info}(P_0, P_1)^2,
\end{align}
for any feasible pairs of $P_t$ and $A_t$.
It is straightforward to verify that the above inequalities become equalities with the given $P_t^{\rm info}$ and $A^{\rm info}$. Therefore, the proposition holds.
\end{proof}

\subsection{A fluid-mechanics interpretation}
Note that $\f_{P_t}^{{\rm info},2}(A_t)$ is a special case of $\f_{P_t}^W(A)$ in \eqref{eq:fPtW} when $W=4P_t^{-1}$. It is equal to 
\[
\f_{P_t}^{{\rm info},2}(A_t)=4 \cE_{p_t}(\|\bv_t(\bx)\|_{P_t^{-1}}^2),
\]
with the velocity field given by $\bv_t(\bx)=A_t\bx$.
Thus if the initial distribution $p_0(\bx)$ is Gaussian, so is $p_t(\bx), \forall t\geq0$.
Proposition \eqref{prop:info2} implies that among all the trajectories that connect two Gaussian probability density functions $p_0$ and $p_1$, the lowest weighted-mass-transport cost is obtained by Gaussian density functions whose covariance matrices are equal to $P_t^{\rm info}$. But this optimal trajectory is obtained under the linear constraint of velocity fields. Next, we remove this constraint and show that this trajectory is still optimal.

\begin{thm}\label{thm:InfoOMT}
Given two zero-mean Gaussian probability density functions $p_0(\bx), p_1(\bx)$ on $\mR^n$ with covariance matrices $P_0, P_1\in \mS^{n}_{++}$, respectively. 
Define $p_t^{\rm info}(\bx), \bv_t^{\rm info}(\bx)$ as the minimizer of
\begin{align}\label{eq:InfoOMT}
\underset{p_t,\bv_t}{\inf}&\bigg\{4 \int_0^1 \cE_{p_t}\left(\|\bv_t(\bx)\|_{P_t^{-1}}^2\right) dt\mid\dot p_t+\nabla_\bx\cdot(p_t\bv_t)=0\bigg\}.
\end{align}
Then $p_t^{\rm info}(\bx)$ is zero-mean Gaussian whose covariance matrix is equal to $P_t^{\rm info}$ from \eqref{eq:Ptinfo} and $\bv_t^{\rm info}(\bx)=A^{\rm info}\bx$ almost surely with $A^{\rm info}$ given by \eqref{eq:At_info}.
Moreover, the optimal value is equal to $\dd_{\rm info}(P_0, P_1)^2$.
\end{thm}
\begin{proof}
First, we define 
\begin{align}\label{eq:VCP}
\left[\begin{matrix} V_t & C_t\\ C_t' & P_t \end{matrix}\right]:=\cE_{p_t}\left(\left[\begin{matrix}\bv_t(\bx)\\ \bx \end{matrix} \right] \left[\begin{matrix}\bv_t(\bx)'& \bx' \end{matrix} \right] \right).
\end{align}
Then applying integral by parts to obtain that
\begin{align}\label{eq:Pt_Ct}
&\dot P_t=\int_{\mR^n}\bx\bx' \dot p_t(\bx)d\bx\nonumber\\
&=\int_{\mR^n}-\bx\bx' \nabla \cdot (p_t(\bx) \bv_t(\bx)) d\bx=C_t+C_t'.
\end{align}
Therefore, the following optimization problem
\begin{align}\label{eq:InfoOMT_lb}
\min_{C_t, V_t, P_t} \bigg\{4\int_0^1 \trace(P_t^{-1}V_t) dt \mid& \left[\begin{matrix} V_t & C_t\\ C_t' & P_t \end{matrix}\right]\in \mS_+^{2n\times 2n},\nonumber\\
 &\dot P_t=C_t+C_t' \bigg\}
\end{align}
provides a lower bound of \eqref{eq:InfoOMT} because the higher-order moments of the probability density functions are not considerd.
On the other hand, \eqref{eq:min_info2} provides an upper bound of \eqref{eq:InfoOMT} because the velocity field is constrained to satisfy the linear system. 
Then, we show that the two bounds coincide.

Note that $V_t-C_t P_t^{-1} C_t'\in \mS_{+}^{n\times n}$. Thus the optimal $V_t$ of \eqref{eq:InfoOMT_lb} should satisfy that $V_t=C_tP_t^{-1}C_t'$. Therefore, \eqref{eq:InfoOMT_lb} is equal to
\begin{align}\label{eq:InfoOMT_lb2}
\min_{C_t, P_t\in \mS_{+}^{n\times n}} &\bigg\{4\int_0^1 \trace(P_t^{-1}C_tP_t^{-1} C_t') dt \mid \dot P_t=C_t+C_t'\bigg\}.
\end{align}
Note that the constrain $P_t\in \mS_+^{n\times n}$ is automatically satisfied due to the inverse barrier objective function. Therefore, we drop the constraint that $P_t\in \mS_+^{n\times n}$ in the following analysis.

Next, we consider the optimization problem \eqref{eq:InfoOMT_lb2} as an optimal control problem with $C_t$ being matrix-valued control. Then we derive the optimal solution using Pontryagin's minimum principle. A necessary condition for the optimal solution is that it much annihilate the variation of the Hamiltonian
\[
h_1(C_t,P_t,\Pi_t):=4\trace(P_t^{-1}C_t P_t^{-1} C_t')+\trace(\Pi_t (C_t+C_t'))
\]
with respect to the control $C_t$. Here, $\Pi_t$ is a symmetric matrix representing the Lagrange multiplier, i.e. the co-state. By setting the partial derivative of $h_1(\cdot)$ with respect to $C_t$ to zero, we obtain that
\begin{align}\label{eq:Ct_eq}
C_t=-\frac14 P_t\Pi_t P_t,
\end{align}
which provides a necessary condition that the optimal $C_t$ is symmetric. Therefore $C_t=\tfrac12\dot P_t$ and the objective function in \eqref{eq:InfoOMT_lb} becomes $\trace(P_t^{-1}\dot P_t P_t^{-1}\dot P_t )=\g_{P,{\rm Rao}}(\dot P_t)$. Thus, the theorem directly follows from \eqref{eq:Ptinfo_min}. For completeness, we finish the proof based on the Hamiltonian $h_1(\cdot)$ in below.

The optimal $\dot \Pi_t$ must annihilate the partial derivative of $h_1(\cdot)$ with respect to $P_t$. This gives rise to
\begin{align}\label{eq:Pit_eq}
\dot \Pi_t=8P_t^{-1}C_tP_t^{-1}C_t'P_t^{-1}.
\end{align}
Then, substituting \eqref{eq:Ct_eq} into \eqref{eq:Pt_Ct} and \eqref{eq:Pit_eq} to obtain that
\begin{align}
\dot P_t&=-\frac12 P_t \Pi_t P_t,\label{eq:Pt_necessary}\\
\dot \Pi_t&= \frac12 \Pi_t P_t \Pi_t.\nonumber
\end{align}
Note that $\dot P_t\Pi_t+P_t \dot \Pi_t=0$ for all $t$. Hence $P_t\Pi_t$ is constant. We set 
\begin{align}\label{eq:PtPiA}
-\frac14 P_t\Pi_t=A.
\end{align}
Thus \eqref{eq:Ct_eq} is equal to $C_t=AP_t=P_t A'$. Multiplying both sides by $P_t^{-\tfrac12}$ gives that $P_t^{-\tfrac12}C_tP_t^{-\tfrac12}=P_t^{-\tfrac12}A P_t^{\tfrac12}$ which is symmetric for all $t$.
Substituting \eqref{eq:PtPiA} to \eqref{eq:Pt_necessary} to give that
\[
\dot P_t=AP_t+P_tA'.
\]
Therefore, 
\[
P_t=e^{At}P_0e^{A't}.
\]
Multiplying $P_0^{-\tfrac12}$ to both sides to give
\[
P_0^{-\frac12}P_tP_0^{-\tfrac12}=P_0^{-\tfrac12}e^{At}P_0  e^{A't}P_0^{-\tfrac12}=e^{2P_0^{-\tfrac12}AP_0^{\tfrac12}t}.
\]
By setting $t=1$, we solve that
\[
A=\frac12 P_0^{\tfrac12}\log(P_0^{-\frac12}P_1P_0^{-\tfrac12})P_0^{-\tfrac12},
\]
which is equal to $A^{\rm info}$.
Furthermore, from \eqref{eq:AtPt_info}, the corresponding $P_t$ is equal to $P_t^{\rm info}$. Then the optimal covariance matrix in \eqref{eq:VCP} is singular and has rank $n$. Thus the optimal velocity field $\bv_t(\bx)$ is equal to $A^{\rm info}\bx$ almost surely, implying that the corresponding $p_t(\bx)$ is Gaussian. Therefore, the theorem is proved.
\end{proof}

Note that the system matrix $A^{\rm info}$ is constant. Thus both $A^{\rm info}$ and $A_t^{\rm omt}$ have fixed eigenspaces. Next, we introduce a different quadratic form of $A_t$ which leads to system matrices with rotating eigenspace.  

\section{Rotation-linear-system based covariance paths}\label{sec:WLS}

\subsection{Weighted-least-squares cost functions}

Note that if $X$ is an asymmetric matrix, i.e. $X'=-X$, then $e^{Xt}$ is a rotation matrix. 
Consequently, if the system matrix $A$ is asymmetric, then the state covariance matrix has a rotating eigenspace.
In this regard, we decompose  
\[
A=A_{\ms}+{A}_{\ma},
\] 
where
$A_{\ms}:=\tfrac12(A+A')$,
$A_{\ma}:=\tfrac12(A-A')$
are the symmetric and asymmetric parts of $A$, respectively.
Then, we define the following weighted-least-squares (WLS) function
\begin{align}\label{eq:feps}
\f_{\epsilon}(A):=\|A_{\ms}\|_{\rm F}^2+\epsilon\|A_{\ma}\|_{\rm F}^2,
\end{align}
where the scalar $\epsilon>0$ weighs the relative significance of symmetric and asymmetric parts of $A$.
If $A$ satisfies $\dot P=AP+PA'$ for a given pair $\dot P$ and $P$, then $\f_{\epsilon}(A)$ is considered as a quadratic form of the non-commutative division of $\dot P$ by $P$, similar to the Fisher-Rao metric. Actually, for scalar-valued covariances, $\f_{\epsilon}(A)$ is equal to the Fisher-Rao metric.

Following \eqref{eq:minA}, we consider the optimal solution to
\begin{align}\label{eq:minA_eps}
\underset{P_t,A_t}{\min} \bigg\{ \int_0^1 \f_{\epsilon}(A_t)dt \mid& \dot P_t=A_t P_t+P_tA_t', P_0, P_1 \mbox{ specified}\bigg\},
\end{align}
for a given pair of endpoints $P_0, P_1\in \mS^n_{++}$ and a scalar $\epsilon>0$.

\subsection{Optimal covariance paths}
To introduce the solution to \eqref{eq:minA_eps}, we define 
\begin{align}\label{eq:TA}
T_{\epsilon,t}(A):=e^{(1+\epsilon) A_{\ma} t}e^{\left( A_{\ms} +\epsilon  A_{\ma} '\right)t}.
\end{align}
The next lemma shows that $T_{\epsilon,t}(\cdot)$ is equal to the state transition matrix of a linear time-varying system.
\begin{lemma}\label{lemma:transition}
Given $A\in \mR^{n\times n}$ and a scalar $\epsilon$. Define 
\begin{align*}
A_{\epsilon,t}&:=e^{(1+\epsilon)A_{\ma}t}A e^{(1+\epsilon)A_{\ma}'t}.\\
\end{align*}
Then
\begin{align}\label{eq:lemma_transition}
\dot T_{\epsilon,t}(A)=A_{\epsilon,t}  T_{\epsilon,t}(A).
\end{align}
\end{lemma}
\begin{proof}
\begin{align*}
\dot T_{\epsilon,t}(A)=&\,e^{(1+\epsilon) A_{\ma} t}((1+\epsilon) A_{\ma}) e^{\left( A_{\ms} +\epsilon  A_{\ma} '\right)t}+e^{(1+\epsilon) A_{\ma} t}(A_{\ms} +\epsilon  A_{\ma} ') e^{\left( A_{\ms} +\epsilon  A_{\ma} '\right)t}\\
=&\,e^{(1+\epsilon) A_{\ma}t} A e^{\left( A_{\ms} +\epsilon  A_{\ma} '\right)t}=A_{\epsilon,t}  T_{\epsilon,t}(A).
\end{align*}
\end{proof}

The following corollary is a direct result of Lemma \ref{lemma:transition}.
\begin{cor}
Given $P_0\in \mS^{n}_{++}, A\in \mR^{n\times n}$ and a scalar $\epsilon$. Define
\begin{align*}
P_{\epsilon,t}&:=T_{\epsilon,t}(A)P_0T_{\epsilon,t}(A)'.
\end{align*}
Then the following equation holds
\begin{align}
\dot P_{\epsilon,t}&=A_{\epsilon,t} P_{\epsilon,t}+P_{\epsilon,t}A_{\epsilon,t}.\label{eq:dotPtwsl}
\end{align}
\end{cor}
Next, we are ready to present the solution to \eqref{eq:minA_eps}.
\begin{prop}\label{thm_wls}
Given $P_0, P_1\in \mS^{n}_{++}$ and a scalar $\epsilon>0$. If there exists a $\Pi_0\in \mS^n$ such that
\begin{align}
P_{\epsilon,t}^{\rm wls}&=T_{\epsilon,t}(A_0) P_0T_{\epsilon,t}( A_0)',\label{eq:Ptwslopt}
\end{align}
satisfies that $P_{\epsilon,1}^{\rm wls}=P_1$ with 
\begin{align}\label{eq:minMS}
A_0=-\tfrac12(P_0 \Pi_0+ \Pi_0 P_0)-\tfrac{1}{2\epsilon}(\Pi_0 P_0-P_0\Pi_0),
\end{align}
and $T_{\epsilon,t}(\cdot)$ given by \eqref{eq:TA}, then $P_{\epsilon,t}^{\rm wls}$ is a minimizer of \eqref{eq:minA_eps}.
The corresponding optimal $A_t$ is equal to
\begin{align}
A_{\epsilon,t}^{\rm wls}&= e^{(1+\epsilon)  A_\ma t} A_0 e^{(1+\epsilon)A_\ma' t}.\label{eq:Atwslopt}
\end{align}
\end{prop}

\begin{proof}
Consider \eqref{eq:minA_eps} as an optimal control problem with $A_t$ being matrix-valued control. Then, the Hamiltonian is as follows
\begin{align*}
h_2(A_t, P_t, \Pi_t)=&\tfrac14\|A_t+A_t'\|_{\rm F}^2+\tfrac{\epsilon}{4} \|A_t-A_t'\|_{\rm F}^2+\trace(\Pi_t(A_tP_t+P_tA_t')),\\
=&\tfrac{1+\epsilon}{2}\trace(A_tA_t')+\tfrac{1-\epsilon}{2}\trace(A_tA_t)+\trace(\Pi_t(A_tP_t+P_tA_t')).
\end{align*}
It is necessary that $\dot \Pi_t$ annihilates the partial derivative of $h_2(\cdot)$ with respect to $P_t$, which gives rise to
\begin{align}
\dot \Pi_t=-\Pi_tA_t-A_t'\Pi_t.\label{eq:Pt_a}
\end{align}
Moreover, the partial derivative of $h_2(\cdot)$ with respect to the control $A_t$ vanishes, which leads to
\begin{align}
(A_t+A_t')+\epsilon(A_t-A_t')+2\Pi_t P_t=0.\label{eq:At_a}
\end{align}
Solving $A_t$ from \eqref{eq:At_a} to obtain that
\begin{align}\label{eq:At_b}
A_t=-\tfrac12(P_t\Pi_t+\Pi_tP_t)-\tfrac{1}{2\epsilon}(\Pi_tP_t-P_t\Pi_t).
\end{align}
Then, substituting \eqref{eq:At_b} in \eqref{eq:dotP} and \eqref{eq:Pt_a}, respectively, to obtain
\begin{align*}
\dot P_t&=(-1+\tfrac{1}{\epsilon})P_t\Pi_tP_t-(\tfrac12+\tfrac{1}{2\epsilon})(\Pi_tP_t^2+P_t^2\Pi_t),\\
\dot \Pi_t&=(1-\tfrac1{\epsilon})\Pi_t P_t\Pi_t+(\tfrac12+\tfrac{1}{2\epsilon})(\Pi_t^2P_t+P_t\Pi_t^2).
\end{align*}
Next, it can be verified that $\dot{(\Pi_tP_t)}-\dot{(P_t\Pi_t)}=0$. 
Thus, the asymmetric part of $A_t$, which is equal to
\[
(A_{t})_{\ma}=-\tfrac{1}{2\epsilon}(\Pi_tP_t-P_t\Pi_t),
\]
is constant and denoted by $A_\ma$.
Taking the derivative of its symmetric part 
\[
(A_{t})_{\ms}=\tfrac12(A_t+A_t')=-\tfrac12(P_t\Pi_t+\Pi_tP_t)
\] 
gives that
\begin{align*}
\dot {(A_{t})_\ms}&=-\tfrac12 \dot{(P_t\Pi_t)}-\tfrac12 \dot{(\Pi_tP_t)}\\
&=\tfrac{1+\epsilon}{2\epsilon}(P_t\Pi_t^2 P_t-\Pi_t P_t^2\Pi_t)\\
&=(1+\epsilon) \big(A_\ma (A_t)_\ms+(A_{t})_\ms A_\ma'\big).
\end{align*}
Since $A_\ma$ is constant, the solution to the above equation is equal to 
\[
(A_t)_\ms =e^{(1+\epsilon)A_\ma t} A_{\ms} e^{(1+\epsilon)A_\ma't}.
\]
Therefore, the optimal $A_t$ has the form
\begin{align*}
A_t&=e^{(1+\epsilon)A_\ma t} A_{\ms} e^{(1+\epsilon)A_\ma't}+A_\ma,\\
&=e^{(1+\epsilon)A_\ma t}A e^{(1+\epsilon)A_\ma't},
\end{align*}
with $A=A_\ma+A_\ms$ being the initial value of $A_t$.
Next, we define a new variable
\begin{align}\label{eq:hatPt}
\hat P_t:=e^{(1+\epsilon)A_\ma 't}P_te^{(1+\epsilon)A_\ma t},
\end{align}
whose derivative is equal to
\begin{align*}
\dot{\hat{P_t}}=&~e^{(1+\epsilon)A_\ma 't}(A_tP_t+P_tA_t')e^{(1+\epsilon)A_\ma t}\\
&+(1+\epsilon)A_\ma' \hat P_t+\hat P_t(1+\epsilon)A_\ma\\
=&~(A_\ms+\epsilon A_\ma') \hat P_t+\hat P_t (A_\ms+\epsilon A_\ma).
\end{align*}
Thus the solution to $\hat P_t$ is equal to
\[
\hat{P_t}=e^{(A_\ms+\epsilon A_\ma')t} P_0 e^{(A_\ms+\epsilon A_\ma)t}.
\]
Substituting this solution to \eqref{eq:hatPt} to obtain that the optimal $P_t$ has the form
\[
P_t=e^{(1+\epsilon)A_\ma t}e^{(A_\ms+\epsilon A_\ma')t} P_0e^{(A_\ms+\epsilon A_\ma)t}e^{(1+\epsilon)A_\ma't}.
\]
In a similar way, we define $\hat \Pi_t=e^{(1+\epsilon)A_\ma 't}\Pi_te^{(1+\epsilon)A_\ma t}$. Then
\[
\dot{\hat{\Pi}}_t=(-A_\ms+\epsilon A_\ma') \hat \Pi_t+\hat \Pi_t (-A_\ms+\epsilon A_\ma),
\] 
whose solution is equal to
\begin{align}\label{eq:Pit_exp}
\Pi_t=e^{(1+\epsilon)A_\ma t}e^{(-A_\ms+\epsilon A_\ma')t} \Pi_0 e^{(-A_\ms+\epsilon A_\ma)t}e^{(1+\epsilon)A_\ma't}.
\end{align}
If \eqref{eq:minMS} holds, then $A_\ms+\epsilon A_\ma'=-P_0\Pi_0$. It is straightforward to show that \eqref{eq:At_b} holds for all $t>0$ for the provided expressions for $A_t, P_t, \Pi_t$.
In this case, $\f_\epsilon(A_t)=\|A_\ms\|_{\rm F}^2+\epsilon \|A_\ma\|_{\rm F}^2$ for all $t$. Therefore, if the $A$ is equal to $\hat A$ in \eqref{eq:minMS}, then the proposed trajectories $A_{\epsilon,t}^{\rm wls}$ and $P_{\epsilon, t}^{\rm wls}$ is local minimizer, which completes the proof.
\end{proof}

Next, we provide an upper bound of the optimal value of \eqref{eq:minA_eps} for all $\epsilon>0$. For this purpose, we define
\[
\hat A= \log (P_0^{-\tfrac12}(P_0^{\tfrac12}P_1P_0^{\tfrac12})^{\tfrac12}P_0^{-\tfrac12}),
\]
which is symmetric. Moreover, $P_t=e^{\hat At}P_0e^{\hat At}$ is a feasible solution to \eqref{eq:minA_eps}. Therefore, the following proposition holds
\begin{prop}
Given $P_0, P_1\in \mS^n_{++}$ and a scalar $\epsilon>0$. Then the optimal value of \eqref{eq:minA_eps} is not larger than $\| \log (P_0^{-\tfrac12}(P_0^{\tfrac12}P_1P_0^{\tfrac12})^{\tfrac12}P_0^{-\tfrac12})\|_{\rm F}^2$. 
\end{prop}

\subsection{On the existence and uniqueness of rotation-system based paths}
We will analyze the existence of a covariance path of the form \eqref{eq:Ptwslopt} that connects two given $P_0, P_1\in \mS^n_{++}$.
To simplify notations, we denote $\alpha=(1+\epsilon)/(2\epsilon)$. Furthermore, we remove the constraint that $\epsilon>0$ and consider all $\epsilon\neq 0$. 
Based on the new parameter $\alpha$, the closed-form equations \eqref{eq:Ptwslopt}, \eqref{eq:minMS} and \eqref{eq:Atwslopt} have defined the following mapping
\begin{align}
h_{\alpha, P_0} (\Pi)&~:~ \mS^{n} \rightarrow \mS^{n}_{++}\nonumber\\
&~:~ \Pi \mapsto e^{\alpha (P_0\Pi-\Pi P_0)}e^{-P_0\Pi} P_0 e^{-\Pi P_0}e^{\alpha(\Pi P_0-P_0\Pi )}.\label{eq:halapha}
\end{align}
We will analyze the existence of $\Pi$ that satisfies 
\begin{align}\label{eq:halpha}
h_{\alpha, P_0} (\Pi)=P_1,
\end{align}
for a given $\alpha$. 
If there exists a $\Pi\in \mS^n$ that solves \eqref{eq:halpha} with $\alpha\neq \frac12$, then there exists a covariance path of the form \eqref{eq:Ptwslopt} with $\epsilon=1/(2\alpha-1)$.

In the special case when $\alpha=0$, \eqref{eq:halpha} becomes $e^{-P_0\Pi}P_0 e^{-\Pi P_0}=P_1$, which is equivalent to
\[
\exp(-2P_0^{\tfrac12}\Pi P_0^{\tfrac12})=P_0^{-\tfrac12}P_1P_0^{-\tfrac12}.
\]
Clearly, it has a unique solution given by
\begin{align}\label{eq:Pi0}
\Pi_0=-\tfrac12 P_0^{-\tfrac12} \log(P_0^{-\tfrac12}P_1P_0^{-\tfrac12}) P_0^{-\tfrac12}.
\end{align}
It is interesting to note that the corresponding covariance path is equal to the Fisher-Rao-based geodesics $P_{t}^{\rm info}$ given by \eqref{eq:Ptinfo}. 

Because the mapping $h_{\alpha, P_0}$ is continuous in term of $\alpha$ and $\Pi$, it is expected that \eqref{eq:halpha} also has a solution if $\alpha$ is sufficiently close to zero.
Specifically, we assume that exists a solution $\Pi$ to \eqref{eq:halpha} for a specific $\alpha$, e.g. $\alpha=0$. Then, we can derive the following expression from \eqref{eq:halpha}
\begin{align}\label{eq:Pialpha}
\Pi  =-\frac12 P_0^{-\tfrac12} \log(P_0^{-\tfrac12}e^{\alpha (\Pi P_0-P_0 \Pi)} P_1 e^{\alpha(P_0 \Pi -\Pi P_0 )}P_0^{-\tfrac12})P_0^{-\tfrac12}.
\end{align}
Next, we will apply perturbation analysis to the above equation to understand the solutions associated with different values for $\alpha$.
Specifically, let $\delta_\alpha$ and $\Delta_\Pi$ denote perturbations to $\alpha$ and $\Pi$, respectively, so that $\alpha+\delta_\alpha$ and $\Pi+\Delta_\Pi$ still satisfy \eqref{eq:Pialpha}. Then, for small perturbations, the perturbation\footnote{
For $A, \Delta\in \mR^{n\times n}$,  
\[
e^{A+\Delta}=e^A+ M_{e^A}(\Delta) + o(\|\Delta\|),
\]
where $M_X(\Delta)$ denotes the non-commutative multiplication of $\Delta$ by $X$ which is defined as
\begin{align}\label{eq:MX}
M_X(\Delta)=\int_0^1 X^{1-\tau} \Delta X^{\tau} d\tau.
\end{align}
For positive definite matrices $A, A+\Delta\in \mS^n_{++}$, 
\[
\log(A+\Delta)=\log(A)+M_{A}^{-1}(\Delta) + o(\|\Delta\|),
\]
where $M_X^{-1}(\Delta)$ denotes the non-commutative devision of $\Delta$ by $X$ which is defined as
\begin{align}\label{eq:MXinv}
M_X^{-1}(\Delta)=\int_0^\infty (X+\tau I)^{-1} \Delta (X+\tau I)^{-1} d\tau.
\end{align}
}  of both sides of \eqref{eq:Pialpha} gives rise to
\begin{align}\label{eq:perturbation}
\Delta_\Pi =-\frac12 P_0^{-\tfrac12}&M_{Q}^{-1}\bigg(\alpha P_0^{-\tfrac12}M_{U}(\Delta_\Pi P_0-P_0\Delta_\Pi)P_1 U' P_0^{-\tfrac12} \nonumber\\
&+\alpha P_0^{-\tfrac12}U P_1 M_{U'}(P_0 \Delta_\Pi-\Delta_\Pi P_0) P_0^{-\tfrac12} \nonumber\\
&+\delta_\alpha  P_0^{-\tfrac12}(\Pi P_0-P_0 \Pi)UP_1 U'P_0^{-\tfrac12}\nonumber\\
&+\delta_\alpha  P_0^{-\tfrac12}UP_1 U'(P_0 \Pi-\Pi P_0) P_0^{-\tfrac12}   \bigg)P_0^{-\tfrac12}+o(|\delta_\alpha|)+o(\|\Delta_\Pi\|),
\end{align}
where $o(|\delta_\alpha|)+o(\|\Delta_\Pi\|)$ denotes higher order terms of the perturbations, 
\begin{align*}
U&=e^{\alpha (\Pi P_0-P_0 \Pi)},\\
Q&=P_0^{-\tfrac12}UP_1 U' P_0^{-\tfrac12},
\end{align*}
and $M_X^{-1}(\cdot)$ and $M_X(\cdot)$ are defined in \eqref{eq:MX} and \eqref{eq:MXinv}, respectively.
Next, we combine all the terms containing $\Delta_\Pi$ on the right hand side of \eqref{eq:perturbation} to define the following linear mapping $\hat h_{\alpha, P_0, \Pi} : \mS^n \rightarrow \mS^n$,
\begin{align}\label{eq:hath}
\hat h_{\alpha, P_0, \Pi} : \Delta_\Pi \mapsto  -\frac{1}{2} P_0^{-\tfrac12}&M_{Q}^{-1}\bigg(P_0^{-\tfrac12}M_{U}(\Delta_\Pi P_0-P_0\Delta_\Pi)P_1 U' P_0^{-\tfrac12} \nonumber\\
&+P_0^{-\tfrac12}U P_1 M_{U'}( P_0 \Delta_\Pi-\Delta_\Pi P_0) P_0^{-\tfrac12}  \bigg)P_0^{-\tfrac12}.
\end{align}
Then the terms that containing the $\delta_\alpha$ is equal to $\delta_\alpha \hat h_{\alpha, P_0, \Pi} (\Pi)$,
Thus, \eqref{eq:perturbation} is equivalent to
\begin{align}
(I-\alpha \hat h_{\alpha, P_0, \Pi})(\Delta_\Pi)=\delta_\alpha \hat h_{\alpha, P_0, \Pi} (\Pi)+o(|\delta_\alpha|)+o(\|\Delta_\Pi\|),
\end{align}
where $I$ denotes the identity mapping. 

Let $\alpha_\tau= \alpha \tau$ denote a smooth trajectory on the interval $\tau \in [0, 1]$ for a given $\alpha$.
If the linear mapping $I-\alpha \hat h_{\alpha_\tau, P_0, \Pi_\tau}$ is invertible, then solution to the following differential equation 
\begin{align}\label{eq:DiffPi}
\frac{d}{d\tau} \Pi_\tau=  (I-\alpha_\tau \hat h_{\alpha_\tau, P_0, \Pi_\tau})^{-1}\circ \alpha  \hat h_{\alpha_\tau, P_0, \Pi_\tau} (\Pi_\tau),
\end{align}
at $\tau=1$ with the initial value given by $\Pi_0$ in \eqref{eq:Pi0} is equal to the unique solution to \eqref{eq:halpha}. 

The following theorem provides a sufficient condition on the existence and uniqueness of the solution. 
To introduce the results, we let $\lambda_{\rm min}(P)$ and $\lambda_{\rm max}(P)$ denote the smallest and the largest eigenvalues of a matrix $P\in \mS^n$. Moreover, we define the following pseudo-norm of a matrix $P\in \mS^n$:
\begin{align}\label{eq:Norms}
\|P\|_{\ma}:=\max_{\Delta\in \mS^n, \Delta\neq 0}\frac{\|\Delta P-P\Delta\|}{\|\Delta\|}.
\end{align}
Note that if $P$ is equal to the identify matrix scaled by any scalar then $\|P\|_{\ma}=0$. If $\|P\|_\ma=0$, we follow the conventions to define $\lambda/\|P\|_\ma=+\infty$ for any $\lambda>0$.
Then, we obtain the following theorem.

\begin{thm}\label{prop:lowalpha}
For a pair of positive definite matrices $P_0, P_1\in \mS^n_{++}$, if $\alpha$ is a scalar such that
\begin{align}\label{eq:alphaBound}
|\alpha| <  \max\left\{ \frac{\lambda_{\rm min}(P_0) \lambda_{\rm min}(P_1)}{\| P_0\|_\ma \lambda_{\rm max}(P_1)}, \frac{\lambda_{\rm min}(P_0) \lambda_{\rm min}(P_1)}{\| P_1\|_\ma \lambda_{\rm max}(P_0)}\right\},
\end{align}
then there exists a unique $\Pi\in \mS^n$ that satisfies 
\begin{align*}
e^{\alpha (P_0\Pi-\Pi P_0)}e^{-P_0\Pi} P_0 e^{-\Pi P_0}e^{\alpha(\Pi P_0-P_0\Pi )}=P_1.
\end{align*}
\end{thm}
\begin{proof}
Following \eqref{eq:DiffPi}, we will prove that $I-\alpha_\tau \hat h_{\alpha_\tau, P_0, \Pi_\tau}$ is invertible if \eqref{eq:alphaBound} holds. It is sufficient to prove that the singular values of the symmetric mapping $\hat h_{\alpha_\tau, P_0, \Pi_\tau}$ are all smaller than 1. For this purpose, we will compute the norm of $h_{\alpha_\tau, P_0, \Pi_\tau}(\Delta_\Pi)$ defined by \eqref{eq:hath}. The norm of the first two terms containing $\Delta_\Pi$ is equal to
\begin{align*}
&\frac12\bigg | \| P_0^{-\tfrac12}M_{Q}^{-1}\bigg(P_0^{-\tfrac12}M_{U}(\Delta_\Pi P_0-P_0\Delta_\Pi)P_1 U' P_0^{-\tfrac12}\bigg)P_0^{-\tfrac12}\|\\
=& \frac12\| \int_0^\infty \int_0^1 P_0^{-\tfrac12} (Q+t_1 I)^{-1} P_0^{-\tfrac12}U^{1-t_2}(\Delta_\Pi P_0-P_0\Delta_\Pi)U^{t_2}P_1 U' P_0^{-\tfrac12}(Q+t_1 I)^{-1} P_0^{-\tfrac12}dt_1dt_2\|\\
\leq &\frac12\int_0^\infty  \|(P_0^{\tfrac12}QP_0^{\tfrac12} +t_1 P_0)^{-1}\|^2 dt_1 \lambda_{\rm max}(P_1) \|P_0\|_\ma \|\Delta_\Pi\|\\
\leq &\, \frac12\int_0^\infty (\lambda_{\rm min}(P_1)+t_1 \lambda_{\min}(P_0))^{-2} dt_1 \lambda_{\rm max}(P_1) \|P_0\|_\ma \|\Delta_\Pi\|\\
\leq &\,\frac12\frac{\| P_0\|_\ma \lambda_{\rm max}(P_1)}{\lambda_{\rm min}(P_0)\lambda_{\rm min}(P_1)}\|\Delta_\Pi\|.
\end{align*}
The same upper bound also holds for the other two terms containing $\Delta$. Combining the bounds for all the four terms lead to 
\[
\|\alpha_\tau \hat h_{\alpha_\tau, P_0, \Pi_\tau}\|\leq \alpha \frac{\|P_0\|_\ma \lambda_{\rm max}(P_1)}{\lambda_{\rm min}(P_0)\lambda_{\rm min}(P_1)}.
\]
Therefore, if 
\begin{align}\label{eq:alpha1}
|\alpha| <  \frac{\lambda_{\rm min}(P_0) \lambda_{\rm min}(P_1)}{\| P_0\|_\ma \lambda_{\rm max}(P_1)},
\end{align}
then $I-\alpha_\tau \hat h_{\alpha_\tau, P_0, \Pi_\tau}$ is invertible for all $\tau\in [0, 1]$. Thus, we have proved the first term on the r.h.s. of \eqref{eq:alphaBound}. 

The second term can be obtained in a similar way by considering a time-reversal path from $P_1$ to $P_0$. Specifically, if
\begin{align}\label{eq:alpha2}
|\alpha| <  \frac{\lambda_{\rm min}(P_0) \lambda_{\rm min}(P_1)}{\| P_1\|_\ma \lambda_{\rm max}(P_0)},
\end{align}
then \eqref{eq:alpha1} implies that there exist a unique $\Pi_1\in \mS^n$ such that the following time-reversal path
\begin{align}\label{eq:Preverse}
 P_{\epsilon, (1-t)}=T_{\epsilon, t}( A_1)P_1 T_{\epsilon, t}( A_1)',
\end{align}
satisfies $ P_{\epsilon, 0}=P_0$, where
\begin{align}\label{eq:Areverse}
A_1=-\tfrac12 (P_1\Pi_1+\Pi_1 P_1)-\tfrac12 (\Pi_1 P_1-P_1 \Pi_1).
\end{align}
Next, we follow \eqref{eq:Pit_exp} to define 
\begin{align}
\Pi_{\epsilon, (1-t)}=e^{(1+\epsilon) (A_1)_\ma t}e^{(- (A_1)_\ms+\epsilon  (A_1)_\ma')t}\Pi_1 e^{(- (A_1)_\ms+\epsilon (A_1)_\ma)t}e^{(1+\epsilon) (A_1)_\ma t},
\end{align}
where $(A_1)_\ms$ and $(A_1)_\ma$ denote the symmetric and asymmetric part of $A_1$, respectively. 
Let $\Pi_0=-\Pi_{\epsilon, (1-t)}$. Then $\Pi_0$ satisfies the conditions in Proposition \ref{thm_wls}. Specifically, the path from \eqref{eq:Preverse} satisfies the time-forward equation
$P_{\epsilon, t}=T_{\epsilon, t}( A_0)P_0 T_{\epsilon, t}( A_0)'$ with $A_0$ given by \eqref{eq:minMS}. Therefore, the proof is complete.
\end{proof}

\subsection{On the computation of local solutions}
If $|\alpha|$ is large so that $I-\alpha_\tau \hat h_{\alpha_\tau, P_0, \Pi_\tau}$ from \eqref{eq:DiffPi} is not invertible for some $\tau \in [0, 1]$, then there may exist multiple solutions to \eqref{eq:halpha}. In below, we propose an approach to compute local solutions \eqref{eq:halpha} based on an approximate initial value.

Specifically, we consider $\hat \Pi_0$ as an initial guess for the solution to \eqref{eq:halpha}. 
We assume that the true endpoint $P_1$ is close to $\hat P_1=h_{\alpha, P_0}(\hat \Pi_0)$. 
By applying perturbation analysis, we obtain that if the pair $\hat \Pi_0+\Delta_\Pi$ and $\hat P_1+\Delta_P$ satisfy \eqref{eq:halpha} then the perturbations should satisfy 
\begin{align*}
\Delta_\Pi-\alpha \hat h_{\alpha, P_0, \hat \Pi}(\Delta_\Pi)+o(\|\Delta_\Pi\|)=-\frac12 P_0^{-\tfrac12}M_{Q}^{-1}(P_0^{-\tfrac12}U \Delta_P U' P_0^{-\frac12}) P_0^{-\tfrac12} +o(\|\Delta_P\|).
\end{align*}
Next, we define a path $P_\tau=(1-\tau) \hat P_1+\tau P_1$ for $\tau \in [0, 1]$.
Then, $P_\tau$ remains in $\mS^n_{++}$ and $\dot P_\tau= P_1-\hat P_1$. 
If $I-\alpha \hat  h_{\alpha, P_0, \hat \Pi}$ is invertible, then the solution to the following differential equation
\begin{align*}
\frac{d}{d\tau} \hat \Pi_\tau=(I-\alpha \hat  h_{\alpha, P_0, \hat \Pi})^{-1}\left(-\frac12 P_0^{-\tfrac12}M_{Q}^{-1}(P_0^{-\tfrac12}U (P_1-\hat P_1) U' P_0^{-\frac12}) P_0^{-\tfrac12}) \right),
\end{align*}
at $\tau=1$ with the initial value given by $\hat \Pi_0$ provides a solution to \eqref{eq:halpha}. 
The solution may depend on the choice of the initial value $\hat \Pi_0$ as illustrated by the following example.

\section{Examples}\label{sec:example}
\subsection{Interpolating covariance matrices}\label{sec:ExampleA}
In this example, we highlight the difference between $P_{\epsilon,t}^{\rm wls}$ and the other two types of trajectories, i.e. $P_{t}^{\rm omt}$, $P_{t}^{\rm info}$, using the following two matrices as the endpoints
\begin{align}\label{eq:P0P1}
P_0=\left[\begin{matrix}1& 0\\ 0& 2 \end{matrix} \right], P_1=\left[\begin{matrix}2& 0\\ 0& 1 \end{matrix} \right].
\end{align}
Applying Eqs. \eqref{eq:Ptinfo} and \eqref{eq:Ptomt} we obtain that
\begin{align*}
P_{t}^{\rm omt}&=\left[\begin{matrix}(1+(\sqrt{2}-1)t)^2& 0\\ 0& (\sqrt{2}+(1-\sqrt{2})t)^2 \end{matrix} \right],\\
P_{t}^{\rm info}&=\left[\begin{matrix}2^t& 0\\ 0& 2^{1-t} \end{matrix} \right],
\end{align*}
which are all diagonal.
On the other hand, if $\epsilon=0$ in \eqref{eq:minMS}, there are infinitely many asymmetric matrices $A_{0}^{\rm wls}$ of the following form 
\[
A_{0}^{\rm wls}=\left[\begin{matrix} 0& \pm \tfrac{(2k+1)\pi}{2}\\ \mp \tfrac{(2k+1)\pi}{2}& 0 \end{matrix} \right],
\]
that makes the objective function equal to zero.
The corresponding covariance paths are equal to
\begin{align*}
P_{0,t}^{\rm wls}=&\left[\begin{matrix}1+\sin^2(\tfrac{(2k+1)\pi}{2}t)& \pm \cos(\tfrac{(2k+1)\pi}{2}t)\sin(\tfrac{(2k+1)\pi}{2}t)\\ \pm \cos(\tfrac{(2k+1)\pi}{2}t)\sin(\tfrac{(2k+1)\pi}{2}t)& 1+\cos^2(\tfrac{(2k+1)\pi}{2}t) \end{matrix} \right],
\end{align*}
which are not diagonal. 

To understand the covariance paths corresponding to nonzero $\epsilon$, we gradually increase $\epsilon$ from $0.001$ to $0.2$ and numerically compute the solution to \eqref{eq:halpha} by using the \emph{fmincon} function in MATLAB$\textsuperscript{\textregistered}$ to search for a symmetric matrix $\Pi$ that minimizes the least-square error $\|h_{\frac{1+\epsilon}{2\epsilon},P_0}(\Pi)-P_1 \|_{\rm F}$.
In this procedure, we apply the minimizer corresponding to a smaller $\epsilon$ as the initial value of the next step with a larger $\epsilon$. In the first step when $\epsilon=0.001$, we choose two initial values for $\hat \Pi$ as 
\[
\hat \Pi_\pm=\pm \left[\begin{matrix} 0& \pm \pi \\ \pm \pi& 0 \end{matrix} \right]\times 10^{-3},
\]
respectively, so that the initial system matrices from \eqref{eq:minMS} are approximately assymetric. 
As $\epsilon$ increases, we obtain two branches of numerical minimizers whose residuals are around $10^{-9}$. 

Figure \ref{fig:Example1Fig1} illustrates the trajectories of the entries of $P_{\epsilon,t}^{\rm wls}$ corresponding to three positive values of $\epsilon$. 
The dashed and solid red lines illustrate the off-diagonal entry of the two branches of minimizers.
As $\epsilon$ increases, the magnitude of the off-diagonal entry reduces. 
Fig. \ref{fig:Example1Fig2} illustrates the off-diagonal entry $A_\ma(1,2)$ of the two branches of local minimizers at different $\epsilon$. The two branches collapse into a unique one when $\epsilon$ passes a threshold value around $0.13$.

\begin{figure*}[htb]
\centering
\includegraphics[width=.3\textwidth]{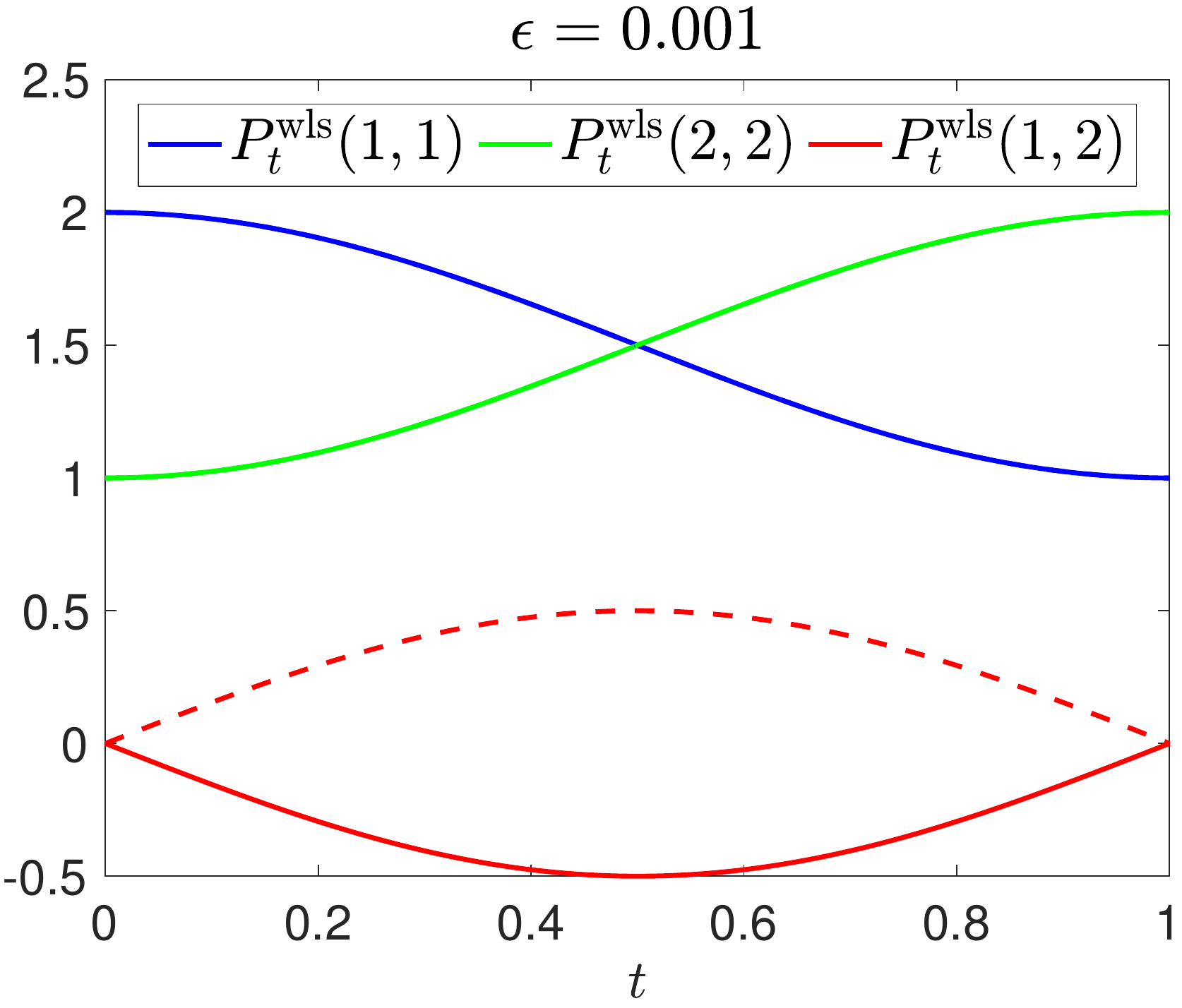}
\includegraphics[width=.3\textwidth]{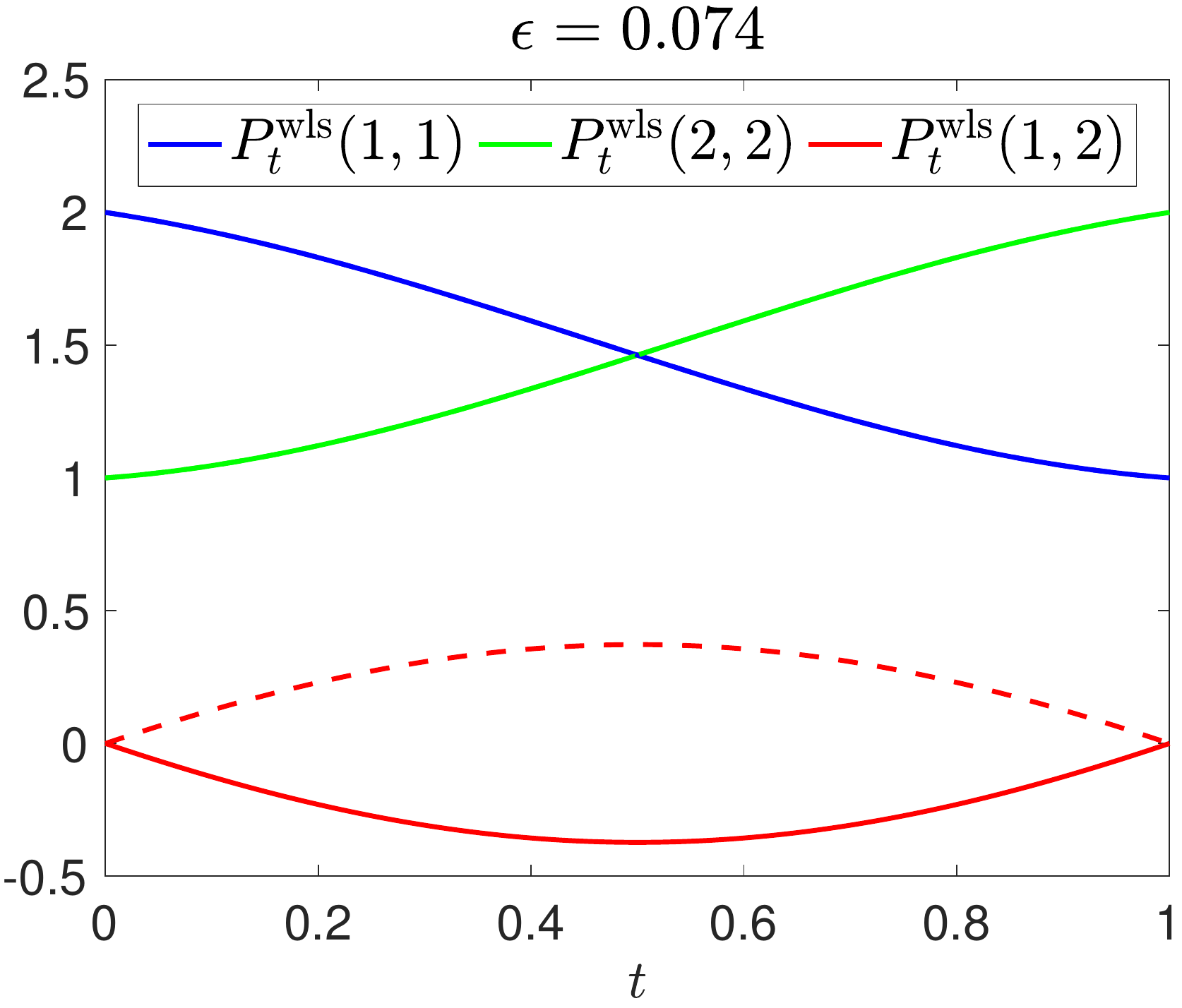}
\includegraphics[width=.3\textwidth]{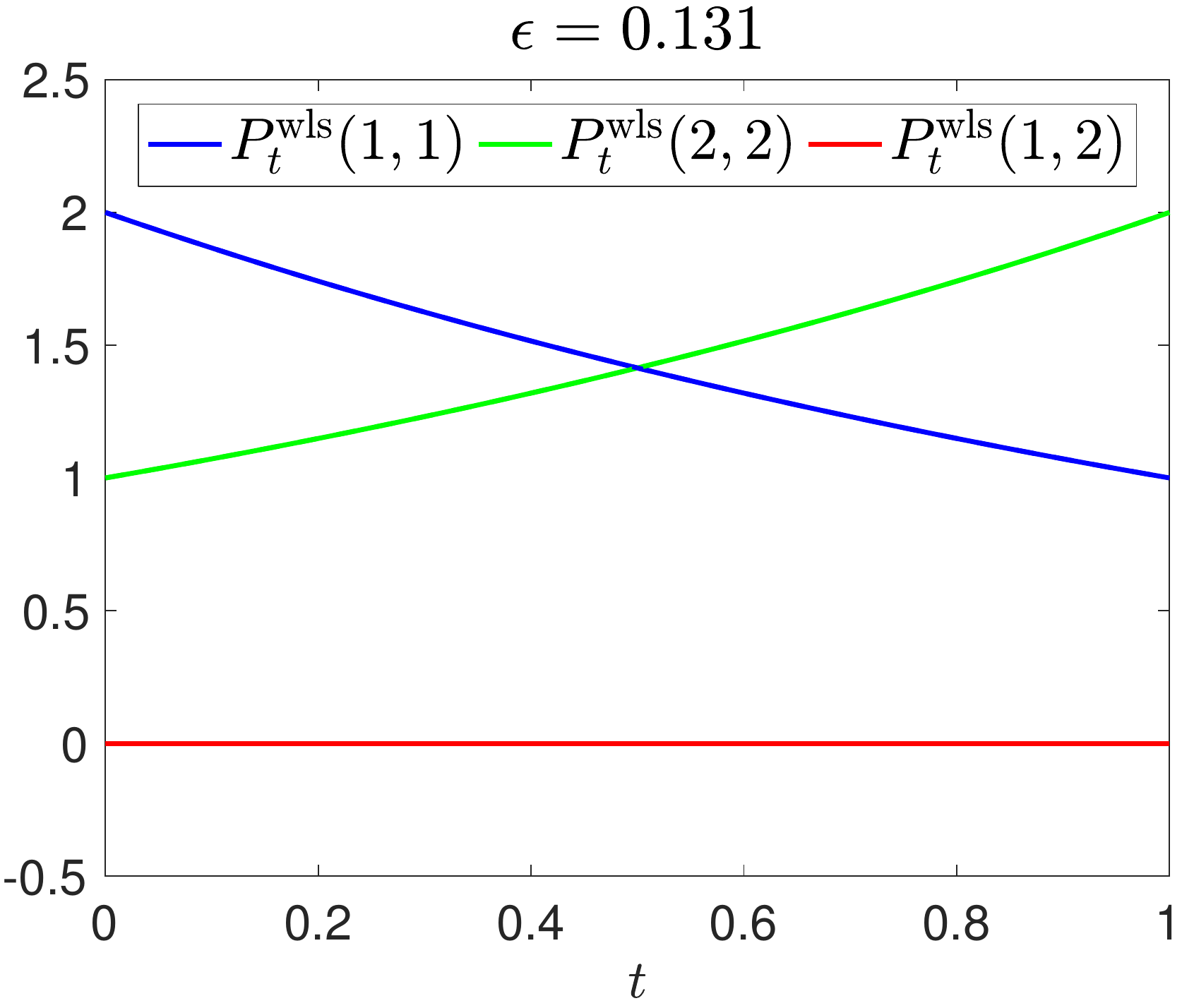}
\caption{\small{An illustration of covariance paths $P_{\epsilon,t}^{\rm wls}$ connecting $P_0$ and $P_1$ in \eqref{eq:P0P1} at three different  values for $\epsilon$.}}\label{fig:Example1Fig1}
\end{figure*}

\begin{figure*}[htb]
\centering
\includegraphics[width=.4\textwidth]{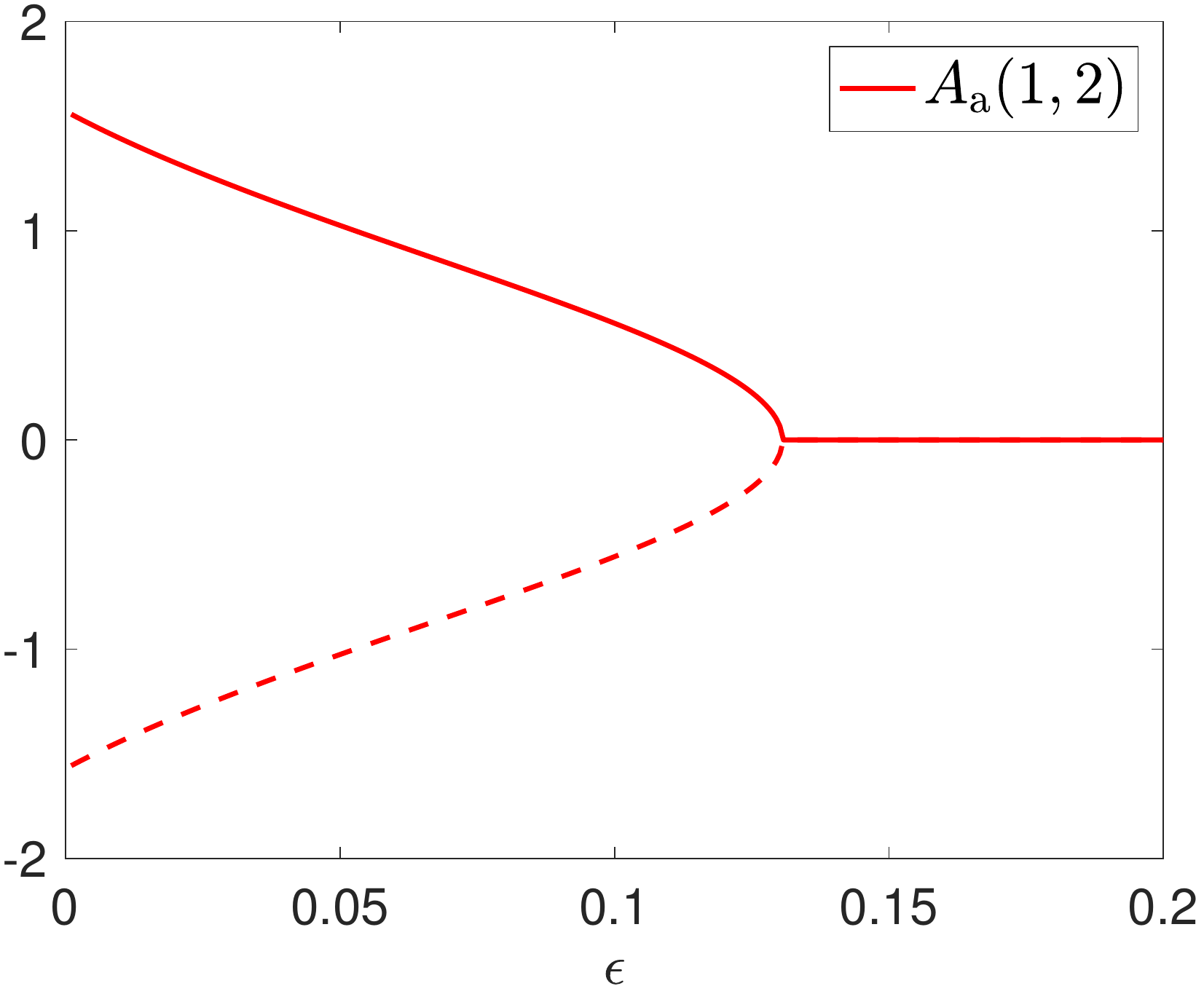}
\caption{\small{An illustration of the off-diagonal entry $A_\ma(1,2)$ of two branches of local minimizers at different $\epsilon$.}}\label{fig:Example1Fig2}
\end{figure*}

\subsection{Regularization of sample covariances}
We investigate an application to use the proposed covariance paths to fit noisy sample covariance matrices from a rsfMRI dataset.
Below we provide detailed descriptions about the data, the method and experimental results.

\subsubsection{Data}
The sample covariances matrices are computed using a rsfMRI data set from the Human Connectome Project \cite{VANESSEN2013}. 
This dataset consists of 1200 rsfMRI image volumes measured in a 15-minute time window. 
The provided data has already been processed by the ICA-FIX method \cite{SMITH2013144}.
It is further processed using global signal regression (GSR) as suggested in \cite{Fox2009}. 
Then we apply the label map from \cite{Yeo2011} to separate brain cortical surface into 7 non-overlapping regions. 
The data sequences from each region are averaged into a one-dimensional time series, providing a 7-dimensional time series sampled at 1200 time points.
The same dataset and preprocessing method have been used in our early work \cite{NingCausality}.
Next, we normalize each dimension of the time series by its standard devision. 
The normalized time series is denoted by $\{\bx_t, t=1, \ldots, 1200 \}$.
Moreover, we split the entire sequence into 10 equal-length segments and compute the corresponding sample covariance matrices as
\[
\tilde P_k=\frac{1}{120} \sum_{i=1}^{120} \bx_{120\times k+i}\bx_{120\times k+i}', \mbox{ for } k=0, \ldots, 9.
\]
Then the time-scale is changed so that $\tilde P_{t_k}$ is equal to $\tilde P_{k}$ with $t_0=0$ and $t_9=1$.
The color arrays in the first row of Fig. \ref{fig:Example2Fig1} illustrate several representative $\tilde P_{t_k}$ at $t=0, \tfrac13, \tfrac23, 1$, respectively.
These figures show that $\tilde P_{t_k}$ has significant fluctuations which is consistent to the observations from \cite{CHANG201081,PRETI201741}.
The main goal of this proof-of-concept experiment is to use the proposed covariance paths to fit these sample covariances and compare their differences. 
The neuroscience aspects of this experiment will not be discussed in this paper.

\subsubsection{Method}

We solve optimization problems of the following form
\begin{align}\label{eq:minPt}
\min_{P_t\in \cP} \sum_{k=0}^K \|P_{t_k}-\tilde P_{t_k}\|_{\rm F}^2,
\end{align}
to obtain smooth paths that fit the measurements, where $K=9$ and $\cP$ represents a suitable set of smooth paths. 
Based on results from the previous sections, we propose three sets of parametric models for the smooth paths which are described in below.

Based on Proposition \ref{thm:transp}, we define 
\begin{align*}
\cP_{\rm omt}:=\bigg\{P_t \mid& P_t=(I-t Q)P_0(I-tQ'), \nonumber\\
&P_0\in \mS^{n}_{++}, Q\in \mR^{n\times n}\bigg\}.
\end{align*}
Note that $Q$ could be a non-symmetric matrix so that $\cP_{\rm omt}$ contains the OMT-based geodesics in the form of \eqref{eq:Ptomt}. 
We use this more general family of covariance paths in order to obtain better fitting results.
It is also clear from Proposition \ref{thm:transp} that a $P_t\cP_{\rm omt}$ is the state covariance of a linear time varying system with $A_t=-Q(I-Qt)^{-1}$.
We apply the \emph{fminsdp} function\footnote{This package is available from \url{https://www.mathworks.com/matlabcentral/fileexchange/43643-fminsdp}.} in MATLAB$\textsuperscript{\textregistered}$ to obtain an optimal solution. The initial values for $P_0$ and $M$ are set to $\tilde P_0$ and the zero matrix, respectively. The same initial values and optimization algorithm are used to solve the subsequent optimization problems. The corresponding optimal value is denoted by $\hat P_t^{\rm omt}$.

The second set of smooth paths is defined based on Proposition \ref{thm:info1} which is given by
\[
\cP_{\rm info}:=\left\{P_t \mid P_t=e^{At}P_0e^{A't}, P_0\in \mS^{n}_{++}, A\in \mR^{n\times n}\right\}.
\]
$\cP_{\rm info}$ includes all geodesic paths in the form of $P_t^{\rm info}$. 
The optimal path in this set is denoted by $\hat P_t^{\rm info}$. 
Clearly, a trajectory in $P_t\in\cP_{\rm info}$ is equal to the state covariance of a linear time-invariant system.

Based on Proposition \ref{thm_wls}, we define
\begin{align*}
\cP_{\epsilon,{\rm wls}}:=\bigg\{P_t \mid& P_t=T_{\epsilon,t}(A)P_0T_{\epsilon,t}(A)', \\
&P_0\in \mS^{n}_{++}, A\in \mR^{n\times n}\bigg\},
\end{align*}
for a given $\epsilon>0$. The corresponding optimal paths are denoted by $\hat P_{\epsilon,t}^{\rm wls}$. 
This set includes all the trajectories that are solutions of \eqref{eq:minA_eps}.
A trajectory in $\cP_{\epsilon,{\rm wls}}$ is equal to the state covariance of a linear time-varying system with the system matrices expressed in the form $e^{(1+\epsilon)  A_\ma t} A e^{(1+\epsilon)A_\ma' t}$. The system matrices corresponding to $\hat P_{\epsilon,t}^{\rm wls}$ is denoted by $\hat A_{\epsilon, t}^{\rm wls}$.
The parameter $\epsilon$ is then searched over a discrete set in $[0,\, 100]$ to minimize fitting errors. Based on the fitting results, we set the value of $\epsilon$ at 20.

\subsubsection{Results}

Figure \ref{fig:Example2Fig2} illustrates the fitting results of 6 representative entries of $\tilde P_{t_k}$. 
The black stars represent the noisy measurements. 
The blue, green, and red plots represent the estimated paths $\hat P_t^{\rm omt}, \hat P_t^{\rm info}$ and $\hat P_{\epsilon,t}^{\rm wls}$, respectively. 
$\hat P_t^{\rm omt}$ and $\hat P_t^{\rm info}$ are very similar with each other. 
Clearly, $\hat P_{\epsilon,t}^{\rm wls}$ has more oscillations which better fits the fluctuations in the measurements. 
The normalized square errors $\left(\sum_{k=0}^K \|\hat P_{t_k}-\tilde P_{t_k}\|_{\rm F}^2\right) /\left(\sum_k^K  \|\tilde P_{t_k}\|_{\rm F}^2\right)$ corresponding to $\hat P_t^{\rm omt}, \hat P_t^{\rm info}, \hat P_{\epsilon,t}^{\rm wls}$ are equal to $0.1683, 0.1671, 0.1543$, respectively. 
Thus, $\hat P_{\epsilon,t}^{\rm wls}$ has the smallest fitting error. 
The overall relative large residual is partly due to the low-signal-to-noise ratio of fMRI data \cite{MURPHY2007565}.
Therefore, the corresponding system matrices $\hat A_{\epsilon, t}^{\rm wls}$ could better explain the dynamic interdependence between brain regions. 
The directed networks in the second row of Fig. \ref{fig:Example2Fig1} illustrates the matrices $\hat A_{\epsilon, t}^{\rm wls}$ at $t=0, \tfrac13, \tfrac23, 1$, respectively. 
The red and blue colors represent positive and negative values, respectively.
The edge widths are weighted by the absolute value of the corresponding entries. 
To simplify visualization, edges with weight smaller than 0.15 are not displayed.

\begin{figure*}[htb]
\centering
\includegraphics[width=1\textwidth]{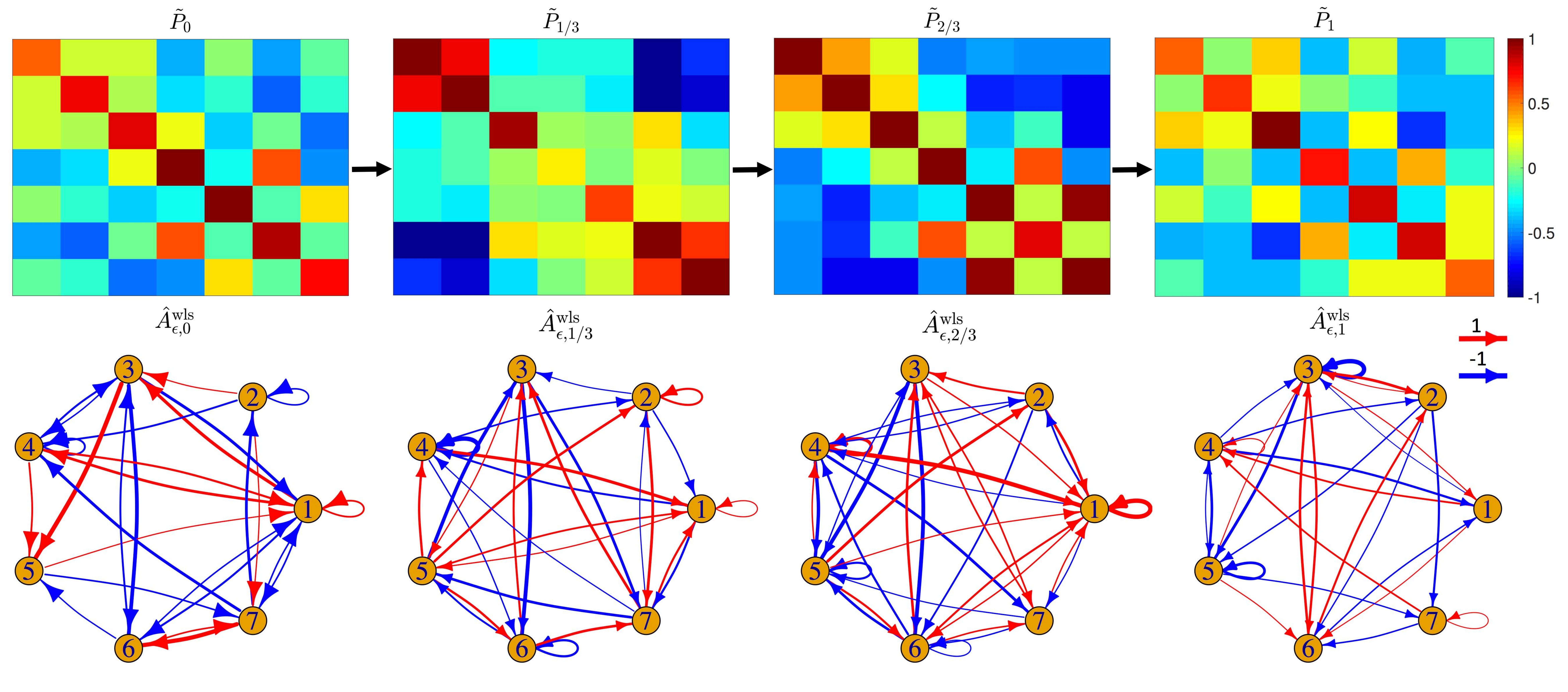}
\caption{\small{The first row illustrates the sample covariance matrices between 7 brain regions computed form different segments of a rsfMRI data set from a human brain. The directed networks in the second row illustrate the estimated system matrices corresponding to the proposed weighted-least-squares trajectories.}}\label{fig:Example2Fig1}
\end{figure*}

\begin{figure*}[htb]
\centering
\includegraphics[width=.7\textwidth]{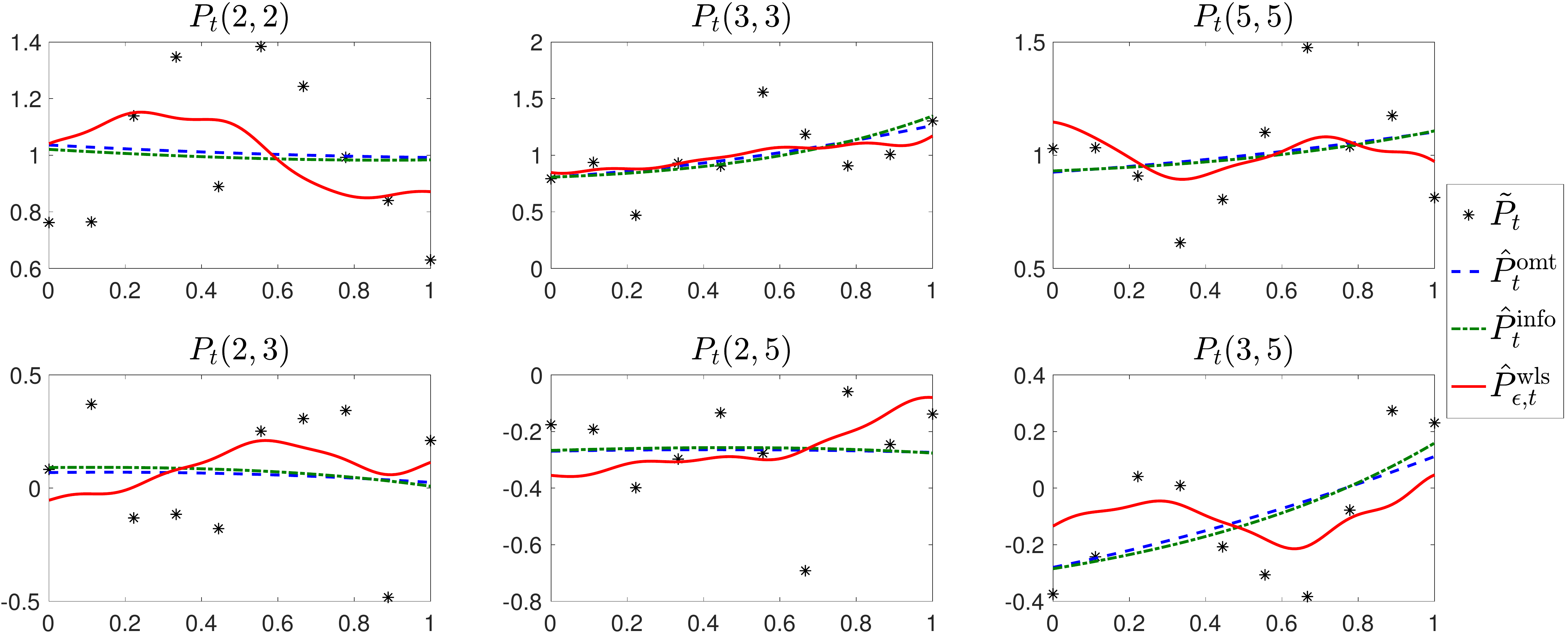}
\caption{\small{The black stars in each image panel illustrate the noisy sample covariances matrices at different time points. The blue, green and red lines are the fitted curves using the proposed three sets of smooth paths.}}\label{fig:Example2Fig2}
\end{figure*}

\section{Discussion}\label{sec:discussion}
In this paper, we have investigated a framework to derive covariance paths on the Riemannian manifold of positive definite matrices by using quadratic forms of system matrices to regularize the path lengths.
We have considered three types of quadratic forms and derived the corresponding covariance paths.
The first and the second quadratic forms lead to the well-known geodesics derived from the Bures-Wasserstein metric from OMT and the Fisher-Rao metric from information geometry, respectively.
In the process, we have derived a fluid-mechanics interpretation of the Fisher-Rao metric in Theorem \ref{thm:InfoOMT}, which provides an interesting weighted-mass-transport view for the Fisher-Rao metric.

The third type of quadric form gives rise to a general family of covariance paths that are steered by system matrices with a rotating eigenspace. 
The rotation velocity is related to the choice of the parameter $\epsilon$.
In the special case when $\epsilon=-1$, i.e. $\alpha=0$, then the eigenspace is not rotating and the trajectories reduce to the Fisher-Rao based geodesics. 
We also analyzed the existence and uniqueness of the paths with sufficiently small $\alpha$.

We note that similar types of trajectories of positive definite matrices with rotating eigenspaces have been investigated in \cite{NingMatrixOMT,Yamamoto2017,ChenOMT} from different angles. This work is developed along similar lines as \cite{Chen2016,Chen2016b}, which focus on the optimal steering of state covariances via linear systems using external input. 
But the approach for developing covariance paths used in this paper is different from early work.

In a proof-of-concept example, we apply three types of smooth paths of state covariance to fit noisy sample covariance matrices from a rsfMRI data set.
A goal of this experiment is to understand directed interactions among brain regions via the estimated system matrices.
As expected, the rotation-system-based covariance path has the best performance in terms of fitting fluctuations in the measurements. 
Therefore, the corresponding system matrices could provide a data-driven tool to understand the structured fluctuations of functional brain activities.
In future work, we will apply this approach to analyze more complex brain networks using different path fitting algorithms. 
Moreover, we will also explore the proposed covariance paths to analyze data from other neuroimaging modalities such as EEG/MEG.   

\section*{Acknowledgment}

{\small
The author would like to thank Tryphon T. Georgiou and Yogesh Rathi for insightful discussions.

This work was supported in part under grants R21MH115280 (PI: Ning), R01MH097979 (PI: Rathi), R01MH111917 (PI: Rathi), R01MH074794 (PI: Westin).

\bibliographystyle{IEEEtran}
\bibliography{IEEEabrv,CovMat}

\end{document}